\documentclass[]{amsart}
\usepackage{amsmath,amsfonts,amsthm,amssymb,enumerate,ifpdf}
\usepackage{graphicx}

\numberwithin{equation}{section}

\newtheorem{theorem}{Theorem}[section]
\newtheorem{corollary}[theorem]{Corollary}

\newtheorem{conjecture}[theorem]{Conjecture}

\newtheorem{lemma}[theorem]{Lemma}

{
\theoremstyle{definition}
\newtheorem{definition}[theorem]{Definition}

\newtheorem{remark}[theorem]{Remark}
\newtheorem{assumption}[]{Assumption}
 
\setlength{\textheight}{43pc}
\setlength {\textwidth}{28pc}

\title[Simplifying branched covering surface-knots]{Simplifying branched covering surface-knots by chart moves involving black vertices}
\author{Inasa Nakamura}
\address{Graduate School of Mathematical Sciences, The University of Tokyo\newline
3-8-1 Komaba, Tokyo 153-8914, Japan \newline
TEL:+81-3-5465-7001}
\email{inasa@ms.u-tokyo.ac.jp}
\subjclass[2010]{Primary 57Q45; Secondary 57Q35}
\keywords{surface-knot; 2-dimensional braid; chart; 1-handle, black vertex}

\address{Current address: Faculty of Electrical, Information and Communication Engineering, 
Institute of Science and Engineering, 
Kanazawa University \newline
Kakumamachi, Kanazawa, 920-1192, Japan \newline
TEL: +81-76-264-5111}
\email{inasa@se.kanazawa-u.ac.jp}

\begin{document}

\begin{abstract}
A branched covering surface-knot is a surface-knot in the form of a branched covering over an oriented surface-knot $F$, where we include the case when the covering has no branch points. A branched covering surface-knot is presented by a graph called a chart on a surface diagram of $F$. 
We can simplify a branched covering surface-knot by an addition of 1-handles with chart loops to a form such that its chart is the union of free edges and 1-handles with chart loops. We investigate properties of such simplifications for the case when branched covering surface-knots have a non-zero number of branch points, using chart moves involving black vertices. 
\end{abstract}

\maketitle

\section{Introduction}\label{sec1}

A {\it surface-knot} is the image of a smooth embedding of a closed connected surface into the Euclidean 4-space $\mathbb{R}^4$ \cite{CKS, Kamada02, Kamada17}. 
We consider oriented surface-knots. 
For a surface-knot $F$, we consider another surface-knot in the form of a branched covering over $F$, called a branched covering surface-knot. Two branched covering surface-knots over $F$ are equivalent if one is taken to the other by an ambient isotopy of $\mathbb{R}^4$ whose restriction to a tubular neighborhood of $F$ is fiber-preserving. 
A branched covering surface-knot over $F$, denoted by $(F, \Gamma)$, is presented by a graph $\Gamma$ called a chart on a surface diagram of $F$. For simplicity, we often 
identify a surface diagram of $F$ with $F$ itself. 

In \cite{N5}, we showed that we can deform a branched covering surface-knot to a simplified form in terms of charts by an addition of 1-handles with chart loops. And in \cite{N6}, we investigated such simplifications, where we included the case of \lq\lq unbranched" covering surface-knots, which is the case when they have no branch points. 
The aim of this paper is to investigate further such simplifications of branched covering surface-knots for the case when they have branch points. In terms of charts, this case is when the charts have degree one vertices called black vertices. 
We use chart moves called CII- and CIII-moves which involve black vertices, and we develop argument especially for branched covering surface-knots with black vertices. Owing to black vertices, we can apply simple deformations especially for this case, which can be considered as a generalization of an elementary deformation of Kamada's charts. 

A 3-disk $h$ embedded in $\mathbb{R}^4$ is called a {\it 1-handle} attaching to $F$ if the intersection $F \cap h$ is a disjoint union of a pair of 2-disks embedded in $\partial h$. The 2-disks in $\partial h$ are called {\it ends} of $h$. 
The surface-knot obtained from $F$ by an {\it addition of a 1-handle} $h$ is the surface

\[
F \cup \partial h\backslash \mathrm{int} (F \cap h),
\]
which is denoted by $F+h$. 
We say a 1-handle $h$ is {\it trivial} if there exists a 3-ball $B^3$ containing $h$ such that $\partial B^3 \cap h$ are the ends of $h$ and $F\cap B^3$ is a 2-disk in $\partial B^3$ containing the ends. 
In this paper, for simplicity, we assume that 1-handles are trivial. Since $F+h$ is orientable, we give $F+h$ the orientation induced from that of $F$.

\begin{assumption}\label{assump1}
In this paper, for simplicity, we assume that 1-handles are trivial, and we assume that $F$ is 
an unknotted surface-knot in the standard form, that is, $F$ is in the form of the boundary of a handlebody in $\mathbb{R}^3 \times \{0\} \subset \mathbb{R}^4$. 
\end{assumption}

For a 2-disk $B^2$ with a boundary point $x \in \partial B^2$ and $I=[0,1]$, we identify a 1-handle $h$ with $B^2 \times I$ such that the ends are $B^2 \times \{0,1\}$. 
Assume that both ends of $h$ are attached to a 2-disk $E$ in $F$. 
We call the {\it core loop} the oriented closed path obtained from $\{x\} \times I \subset \partial h$ by connecting the initial and terminal points $\{ x\} \times \{0,1\}$ by a simple arc in $E$, with the orientation induced from that of $I$. 
We determine the {\it cocore} of $h$ by the oriented closed path $\partial B^2 \times \{0\} \subset h$, with the orientation of $\partial B^2$; see Fig. \ref{fig1}. 
In this paper, for simplicity, we do not distinguish framings of 1-handles. There are two types of framings, presented by the core loop and the cocore as indicated in Fig. \ref{fig1} \cite{Boyle2, Livingston}, see also \cite[Lemma 4.2]{N5} and \cite[Remark 1.1]{N6}. 

\begin{figure}[ht]
\centering
\includegraphics*[height=3.5cm]{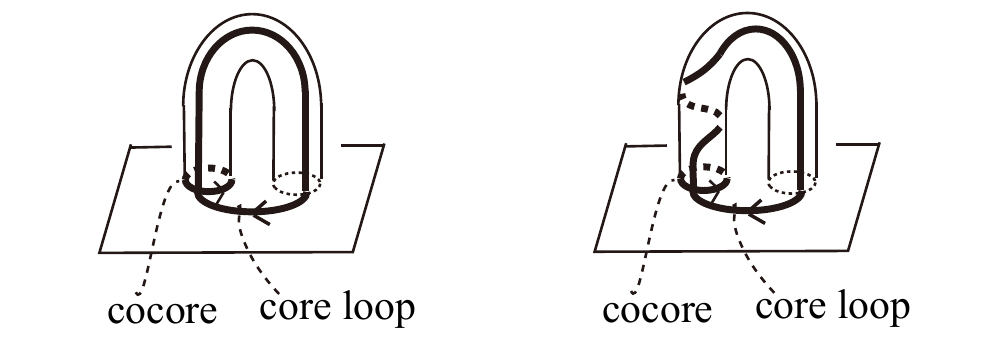}
\caption{The core loop and the cocore of a 1-handle. There are two types.}
\label{fig1}
\end{figure}

 A chart is a finite graph satisfying certain conditions such that each edge is equipped with a label and an orientation. A branched covering surface-knot over $F$ is presented by a chart on $F$. 
A {\it chart loop} is a closed path consisting of a closed edge of a chart or diagonal edges of a chart connected with vertices of degree 4 (crossings). 
 
In this paper, we treat mainly three types of 1-handles: $h(e,e)$, $h(\sigma_i, e)$, and $h(\sigma_i, \sigma_j^\epsilon)$. 
We assume that a 1-handle is attached to a 2-disk. 
We denote by $h(e,e)$ a 1-handle equipped with an empty diagram (empty chart). 
We denote by $h(\sigma_i, e)$ a 1-handle equipped with a chart loop parallel to the core loop with the label $i$ and the orientation coherent with that of the core loop. We denote by $h(\sigma_i, \sigma_j^\epsilon)$ the 1-handle obtained from $h(\sigma_i, e)$ by an addition of a chart loop parallel to the cocore with the label $j$ and the orientation coherent (respectively, incoherent) with that of the cocore if $\epsilon=+1$ (respectively, $-1$). We call each of such 1-handles a {\it 1-handle with chart loops}, or simply a {\it 1-handle}. 

A {\it free edge} is an edge of a chart whose end points are black vertices. 
Let $\Gamma_0$ be an empty chart or a chart consisting of a disjoint union of several free edges on $F$. 
Let $(F, \Gamma_0)$ be the branched covering surface-knot determined by $\Gamma_0$, and let $h(a_1, b_1)$, $\ldots$, $h(a_g, b_g)$ be 1-handles with chart loops. We take mutually distinct embedded 2-disks $E_1$, \ldots, $E_g$ in $F$ such that there are no edges nor vertices of $\Gamma_0$ on the 2-disks, and we attach 1-handles to these disks. 
We denote the branched covering surface-knot which is the result of the 1-handle addition  by $(F', \Gamma')=(F, \Gamma_0) + \sum_{i=1}^g h(a_i,b_i)$. Note that since $\Gamma_0$ is a disjoint union of free edges, the presentation is well-defined.  

We showed in \cite{N5} the following results. 
Under Assumption \ref{assump1}, the results are written as follow. See \cite{N6} for the same notations and terminologies used here. Let $N$ be a positive integer.

\begin{theorem}[{\cite[Theorem 1.6]{N5}}]\label{thm1-2}
Let $(F, \Gamma)$ be a branched covering surface-knot of degree $N$. 
By an addition of finitely many 1-handles in the form $h(\sigma_i, e)$ or $h(e,e)$ $(i \in \{1, \ldots,N-1\})$, to appropriate places in $F$, $(F, \Gamma)$ deforms to 
\begin{equation}\label{eq1-1}
(F, \Gamma_0) +\sum_{k} h(\sigma_{i_k}, e)+\sum_l h(\sigma_{i_l}, \sigma_{j_l}^{\epsilon_l})+\sum h(e,e),
\end{equation}
where $i_k, i_l, j_l \in \{1.\ldots,N-1\}, |i_l-j_l|>1$ and $\epsilon_l\in \{+1, -1\}$, and $\Gamma_0$ is a chart consisting of several (maybe no) free edges. 
\end{theorem}

\begin{theorem}[{\cite[Theorem 1.8]{N5}}]\label{thm1-4}
Let $(F, \Gamma)$ be a branched covering surface-knot of degree $N$.  By an addition of finitely many 1-handles in the form $h(\sigma_i, e)$, $h(\sigma_i, \sigma_j^\epsilon)$ or $h(e,e)$ $(i, j \in \{1, \ldots,N-1\}, |i-j|>1, \epsilon \in \{+1, -1\})$, to appropriate places in $F$, $(F, \Gamma)$  deforms to  
\begin{equation}\label{eq1-3}
(F, \Gamma_0) +\sum_k h(\sigma_{i_k}, e)+\sum h(e,e),
\end{equation}
where $i_k \in \{1.\ldots,N-1\}$ and $\Gamma_0$ is a chart consisting of several free edges. 
\end{theorem}

\begin{definition}\label{def1-9}
We call $(F, \Gamma)$ in the form (\ref{eq1-1}) (respectively, (\ref{eq1-3})) a branched covering surface-knot in a {\it weak simplified form} (respectively, {\it simplified form}), and 
we call the minimal number of 1-handles necessary to deform $(F, \Gamma)$ to the form (\ref{eq1-1}) (respectively, (\ref{eq1-3})) the {\it weak simplifying number} (respectively,  {\it simplifying number}) of $(F, \Gamma)$, denoted by $u_w(F, \Gamma)$ (respectively, $u(F, \Gamma)$).  
\end{definition}

Our results are as follow. Let $(F, \Gamma)$ be a branched covering surface-knot of degree $N$. We denote by $b(\Gamma)$ and $w(\Gamma)$ the numbers of black vertices and white vertices, respectively. 
We consider the case when $b(\Gamma)>0$. 

\begin{theorem}\label{thm4-4}
Let $(F, \Gamma)$ be a branched covering surface-knot of degree $N$ with $b(\Gamma)>0$ and $w(\Gamma)=0$. Then $u_w(F, \Gamma) \leq N-2$ and $u(F, \Gamma) \leq N-2$.
\end{theorem}

Theorem \ref{thm4-4} gives a better estimate than that given in \cite[Corollary 1.10 (1.12)]{N6}, and the proof is much simpler owing to black vertices. 

By the proof of \cite[Theorem 1.7]{N6}, by an addition of $\lfloor w(\Gamma)/2+b(\Gamma)(N-2)/4\rfloor$ 1-handles, $(F, \Gamma)$ is deformed to have no white vertices.  
Thus we have the following corollary. 

\begin{corollary}\label{thm1-7}
Let $(F, \Gamma)$ be a branched covering surface-knot of degree $N$. 
Let  $b(\Gamma)$ and $w(\Gamma)$ be the numbers of black vertices and white vertices, respectively. 
If $b(\Gamma)>0$, then 
\begin{equation}\label{eq2}
u_w(F, \Gamma) \leq \left\lfloor \frac{w(\Gamma)}{2}+\frac{b(\Gamma)}{4}(N-2) \right\rfloor+N-2.
\end{equation}
where $\lfloor x\rfloor$ is the largest integer less than or equal to $x$. 
Further, $(\ref{eq2})$ also holds true for $u(F, \Gamma)$. 
\end{corollary}

Corollary \ref{thm1-7} gives a better estimate than that given in \cite[Corollary 1.10 (1.13)]{N6}. 

By the proof of \cite[Proposition 1.11]{N5}, by an addition of $w(\Gamma)+2c(\Gamma)$ 1-handles, $(F, \Gamma)$ deforms to have no white vertices, where $c(\Gamma)$ is the number of crossings. Thus Theorem \ref{thm4-4} implies $u_{(w)}(F, \Gamma) \leq w(\Gamma)+2c(\Gamma)+N-2$.
Here, we give a better estimate. 

\begin{theorem}\label{thm1-16}
Let $(F, \Gamma)$ be a branched covering surface-knot of degree $N$. 
Let  $b(\Gamma)$ and $w(\Gamma)$ be the numbers of black vertices and white vertices, respectively. 
If $b(\Gamma)>0$, then 
\begin{equation}\label{eq1}
u_w(F, \Gamma) \leq w(\Gamma)+N-2.
\end{equation}
Further, $(\ref{eq1})$ also holds true for $u(F, \Gamma)$. 
\end{theorem}

In \cite[Conjecture 1.13 (1.19)]{N6}, we gave the following conjecture. 

\begin{conjecture}\label{conj}
 For a branched covering surface-knot $(F, \Gamma)$ of degree $N$, 
\begin{equation}\label{eq:conj}
u(F,\Gamma) \leq \max\{ u_w(F, \Gamma), N-1\}. 
\end{equation}
\end{conjecture}

Conjecture \ref{conj} follows for the case $b(\Gamma)>0$ from the following theorem. 

\begin{theorem}\label{thm1-8}
Let $(F, \Gamma)$ be a branched covering surface-knot in a weak simplified form for $u_w(F, \Gamma)$ such that $b(\Gamma)>0$, where $b(\Gamma)$ is the number of black vertices of $\Gamma$. Let $f(\Gamma)=b(\Gamma)/2$, the number of free edges of $\Gamma$, and let $h(F)$ be the number of 1-handles in $(F, \Gamma)$. Then 
\begin{equation}
u(F, \Gamma) \leq \max\{0, N-f(\Gamma)-h(F)-1\}. 
\end{equation}
\end{theorem}

Let $(F, \Gamma)$ be a branched covering surface-knot of degree $N$ with $b(\Gamma)>0$. By an addition of $u_w(F, \Gamma)$ number of 1-handles, $(F, \Gamma)$ deforms to be in a weak simplified form $(F', \Gamma')$. For the number $h(F')$ of 1-handles in $(F', \Gamma')$, we have $u_w(F, \Gamma) \leq h(F')$. 
Since $u_w(F, \Gamma)+N-f(\Gamma')-h(F')-1 \leq N-f(\Gamma')-1<N-1$, Theorem \ref{thm1-8} implies Conjecture \ref{conj}. 
   
The paper is organized as follows. In Section \ref{sec2}, we review branched covering surface-knots and their chart presentations. 
In Section \ref{sec3}, we show Theorems \ref{thm4-4} and \ref{thm1-16}. In Section \ref{sec4}, we show Theorem \ref{thm1-8}.

\section{Branched covering surface-knots (formerly 2-dimensional braids) and their chart presentations}\label{sec2}
In this section, we review a branched covering surface-knot, formerly a 2-dimensional braid over a surface-knot \cite{N4} (see also \cite{N}). We adopted the term \lq\lq branched covering surface-knot'' in \cite{N6}.  A branched covering surface-knot is an extended notion of 2-dimensional braids or surface braids over a 2-disk \cite{Kamada92, Kamada02, Rudolph}. A branched covering surface-knot over a surface-knot $F$ is presented by a finite graph called a chart on a surface diagram of $F$ \cite{N4} (see also \cite{Kamada92, Kamada96, Kamada02}). For two branched covering surface-knots of the same degree, they are equivalent if their surface diagrams with charts are related by a finite sequence of ambient isotopies of $\mathbb{R}^3$, and local moves called C-moves \cite{Kamada92, Kamada02} and Roseman moves \cite{N4} (see also \cite{Roseman}). Here, we only review C-moves.  
 
\subsection{Branched covering surface-knots over a surface-knot}
 
Let $B^2$ be a 2-disk, and let $N$ be a positive integer. 
For a surface-knot $F$, let $N(F)=B^2 \times F$ be a tubular neighborhood of $F$ in $\mathbb{R}^4$. 

\begin{definition}
A closed surface $S$ embedded in $N(F)$ is called a {\it branched covering surface-knot over $F$ of degree $N$} if it satisfies the following conditions. 

\begin{enumerate}[(1)]
\item
The restriction $p|_{S} \,:\, S \rightarrow F$ is a branched covering map of degree $N$, where $p\,:\, N(F) \to F$ is the natural projection. 

\item The number of points consisting $S \cap p^{-1}(x)$ is $N$ or $N-1$ for any point $x \in F$.
\end{enumerate}
Take a base point $x_0$ of $F$. 
Two branched covering surface-knots over $F$ of degree $N$ are {\it equivalent} if there is an ambient isotopy of $\mathbb{R}^4$ whose restriction to $N(F)=B^2 \times F$ is a fiber-preserving ambient isotopy relative to $p^{-1}(x_0)$ which takes one to the other. 

\end{definition}
\subsection{Chart presentation}\label{sec2-2}

A {\it surface diagram} of a surface-knot $F$ is the image of $F$ in $\mathbb{R}^3$ by a generic projection, equipped with the over/under information on sheets along each double point curve. 
 
\begin{definition} 
Let $N$ be a positive integer. 
A finite graph $\Gamma$ on a surface diagram $D$ is called a {\it chart} of degree $N$ if 
it satisfies 
the following conditions.

\begin{enumerate}[(1)]
\item
The intersection of $\Gamma$ and the singularity set of $D$ consists of a finite number of transverse intersection points of edges of $\Gamma$ and double point curves of $D$, which form vertices of degree $2$.

 \item Every vertex has degree $1$, $2$, $4$, or $6$.
 
 \item Every edge of $\Gamma$ is oriented and labeled by an element of 
       $\{1,2, \ldots, N-1\}$. Around vertices of degree $1$, $4$, or $6$, the edges 
are oriented and labeled as shown in 
Fig. \ref{fig2}. 
We depict a vertex of degree 1 by a black vertex, and a vertex of degree 
6 by a white vertex, and we call a vertex of degree $4$ a crossing.
 \end{enumerate}
\end{definition}

\begin{figure}[ht] 
\includegraphics*{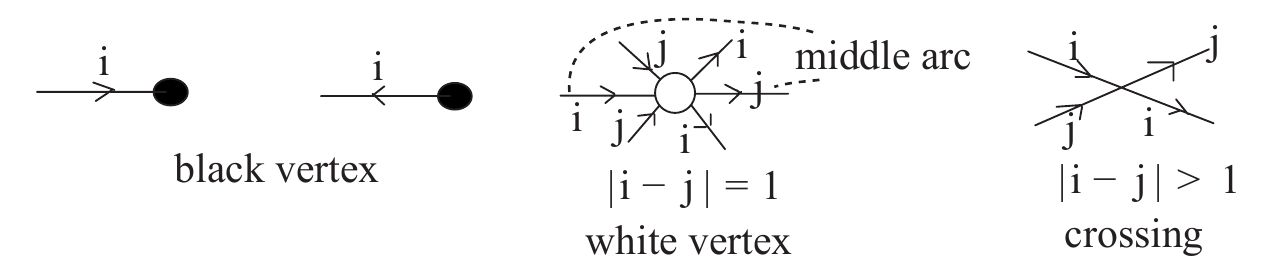}
\caption{Vertices in a chart, where $i \in \{1,\ldots,N-1\}$.}
\label{fig2}
\end{figure}

Remark that since we consider $F$ as an unknotted surface-knot in the standard form, its surface diagram contains no singularities and vertices of degree 2 do not appear. For the definition of vertices of degree 2, see \cite{N4}.  

A branched covering surface-knot over a surface-knot $F$ is presented by a chart $\Gamma$ on a surface diagram of $F$ \cite{N4}. We present such a branched covering surface-knot by $(F, \Gamma)$. 
 \\

We call an edge of a chart a {\it chart edge} or simply an {\it edge}. 
We regard diagonal edges connected with crossings as one edge, and we regard that a crossing is formed by intersections of two edges. 
A chart edge connected with no vertices or a chart edge with crossings 
is called a {\it chart loop} or simply a {\it loop}. A chart edge whose endpoints are black vertices is called a {\it free edge}. A chart is said to be {\it empty} if it is an empty graph. 

In order to distinguish parts of an edge connected with two vertices at the end points, we use the notion of an arc. 
For a vertex $v$, 
we call the intersection of an edge connected with $v$ and a small neighborhood of $v$ an {\it arc}. 
For a white vertex, we call an arc which is the middle of the three adjacent arcs with the coherent orientation a {\it middle arc}, and we call an arc which is not a middle arc a {\it non-middle arc}. Around a white vertex, there are two middle arcs and four non-middle arcs; see Fig. \ref{fig2}. 
\\

We explain the correspondence between a branched covering surface-knot and the chart presentation. 
Let $S$ be a branched covering surface-knot over a surface-knot $F$. 
We explain how to obtain a chart on a 2-disk $B$ in a surface diagram which does not intersect with singularities of $F$. 
We denote the covering surface $S\cap p^{-1}(B)$ by the same notation $S$. We identify a tubular neighborhood $N(B)$ by $I \times I \times B$. Consider the singularity set $\mathrm{Sing}(p_1(S))$ of the image of $S$ by the projection $p_1$ to $I \times B$. Perturbing $S$ if necessary, we assume that $\mathrm{Sing}(p_1(S))$ consists of double point curves, triple points, and branch points. Further, we assume that the singular set of the image of $\mathrm{Sing}(p_1(S))$ by the projection to $B$ consists of a finite number of double points such that the preimages belong to double point curves of $\mathrm{Sing}(p_1(S))$. Thus 
the image of $\mathrm{Sing}(p_1(S))$ by the projection to $B$ forms a finite graph $\Gamma$ on $B$ such that the degree of a vertex of $\Gamma$ is either $1$, $4$ or $6$, where we ignore the points in $\partial B$. An edge of $\Gamma$ corresponds 
to a double point curve, and a vertex of degree $1$ (respectively, $6$) 
corresponds to a branch point (respectively, triple point). 

For such a graph $\Gamma$ obtained from a covering surface $S$, we assign orientations and labels to all edges of $\Gamma$ as follows. Take a path $\rho$ in $B$ such that $\rho \cap \Gamma$ is a point $x$ of an edge $E$ of $\Gamma$. Then $S \cap p^{-1} (\rho)$ is a classical $N$-braid with one crossing in $p^{-1}(\rho)$ such that $x$ corresponds to the crossing of the $N$-braid, where $N$ is the degree of $S$. Let $\sigma_{i}^{\epsilon}$ ($i \in \{1,2,\ldots, N-1\}$, 
$\epsilon \in \{+1, -1\}$) be the presentation of $S \cap p^{-1}(\rho)$. We assign $E$ the label $i$, and the orientation such that 
the normal vector of $\rho$ is coherent  (respectively, incoherent) with the orientation of $E$ if $\epsilon=+1$ (respectively, $-1$), where the normal vector of $\rho$ is the vector $\vec{n}$ such that $(\vec{v}(\rho), \vec{n})$ corresponds to the orientation of $B$ for a tangent vector $\vec{v}(\rho)$ of $\rho$ at $x$. This is the chart presentation of $S\cap p^{-1}(B)$. 

\subsection{C-moves}
  
{\it C-moves} ({\it chart moves}) are local moves of a chart, consisting of three types called CI-moves, CII-moves, and CIII-moves. 
Let $\Gamma$ and $\Gamma^{\prime}$ be two charts of the same degree on a surface diagram $D$. We say $\Gamma$ and $\Gamma'$ are related by a {\it CI-}, {\it CII-} or {\it CIII-move} if there exists a 2-disk $B$ in $D$ such that $B$ does not intersect with the singularities of $D$, and the loop $\partial B$ is in general position with respect to $\Gamma$ and $\Gamma^{\prime}$ and $\Gamma \cap (D-B)=\Gamma^{\prime} \cap (D-B)$, and the following conditions hold true.  
 
(CI) There are no black vertices in $\Gamma \cap B$ nor $\Gamma^{\prime} \cap B$. The moves given in Fig. \ref{fig3} are called CI-M1, CI-M2, CI-M3 moves, respectively. 
See \cite{CKS} for figures of a complete generating set of CI-moves. 

(CII) $\Gamma \cap B$ and $\Gamma' \cap B$ are as in Fig. \ref{fig3}, where $|i-j|>1$.

(CIII) $\Gamma \cap B$ and $\Gamma' \cap B$ are as in Fig. \ref{fig3}, where $|i-j|=1$, and the black vertex is connected to a non-middle arc of a white vertex. 
\\

\begin{figure}[ht]
\includegraphics*[width=13cm]{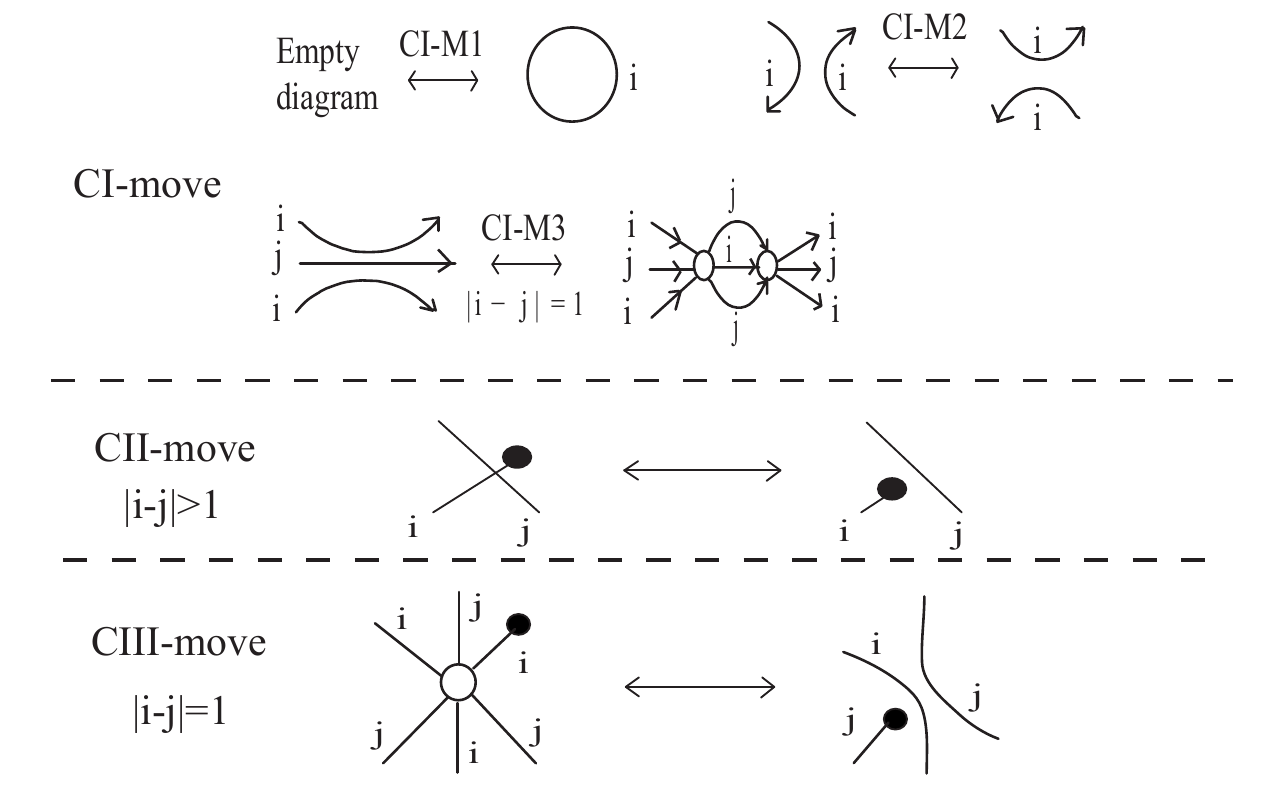}
\caption{Examples of C-moves. For simplicity, we omit orientations of some edges.} 
\label{fig3}
\end{figure}

For charts $\Gamma$ and $\Gamma'$ of the same degree on a surface diagram of a surface-knot $F$, 
their presenting branched covering surface-knots are equivalent if the charts are related by a finite sequence of C-moves \cite{Kamada92, Kamada02}.

\section{Simplifying branched covering surface-knots with black vertices}\label{sec3}

We say that 1-handles with chart loops attached to a 2-disk in a 3-ball $B^3$ are {\it equivalent} if one is carried to the other by an ambient isotopy of $B^3$ and C-moves. Branched covering surface-knots with equivalent 1-handles are equivalent. We use the notation \lq\lq $\sim$'' to denote the equivalence relation. For commutative braids $a, b$, we denote by $h(a,b)$  a 1-handle equipped with a chart without black vertices such that the cocore and the orientation-reversed core loop  presents $a$ and $b$, respectively. 
Unless otherwise said, we assume that 1-handles are attached to a fixed 2-disk such that there are no chart edges nor vertices except those of 1-handles. We denote by $f_i$ a free edge with the label $i$. We denote by the notation \lq\lq $+f_i$'' the branched covering surface-knot obtained from the \lq\lq addition'' of $f_i$, that is, by adding $f_i$ into a 2-disk which has no chart edges nor vertices. 
Let $N$ be the degree of charts.

\begin{lemma}\label{lem3-1}
We have 
\begin{eqnarray} 
&& h(\sigma_i, e) \sim h(\sigma_i^{-1}, e), \\
&& h(e, \sigma_i) \sim h(e, \sigma_i^{-1}), \\
&& h(\sigma_i, \sigma_j^\epsilon) \sim h(\sigma_i^{-1}, \sigma_j^{-\epsilon}),
\end{eqnarray}
where $|i-j|>1$, $i,j \in \{1, \ldots, N-1\}$, and $\epsilon \in \{+1, -1\}$. 
\end{lemma}

\begin{proof}
Rotating the 1-handle and taking the orientation reversal of the original core loop as the new core loop, we have the result. See Fig. \ref{fig4}, and see also \cite[Lemma 4.1]{N5}. 
\end{proof}

\begin{figure}[ht]
\centering
\includegraphics*[height=2.5cm]{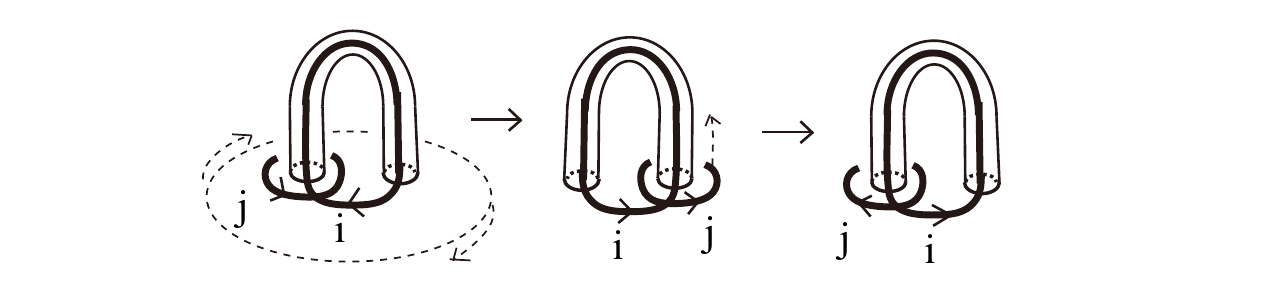}
\caption{$h(\sigma_i, \sigma_j) \sim h(\sigma_i^{-1},\sigma_j^{-1})$, where $|i-j|>1$. }
\label{fig4}
\end{figure}

\begin{lemma}\label{lem2-3}
$(1)$ Let $\rho$ be a chart loop of the label $i$ such that there may be crossings on $\rho$. If we have a free edge $f_i$ of the label $i$,  then we can eliminate $\rho$ using $f_i$. Conversely, if we have a free edge $f_i$ of the label $i$, then we can add a chart loop $\rho$ of the label $i$ such that an arc of $\rho$ is in a neighborhood of $f_i$.

$(2)$ A similar result as $(1)$ holds true for a chart loop $\rho$ with no crossings, and when we take a 1-handle with a chart loop $h(\sigma_i, e)$ instead of $f_i$. 
\end{lemma}

\begin{proof}
Applying a CI-M2 move and moving a black vertex or an end of the 1-handle as indicated in Fig. \ref{fig5}, together with CII-moves, we have the required result. 
\begin{figure}[ht]
\centering
\includegraphics*[height=5cm]{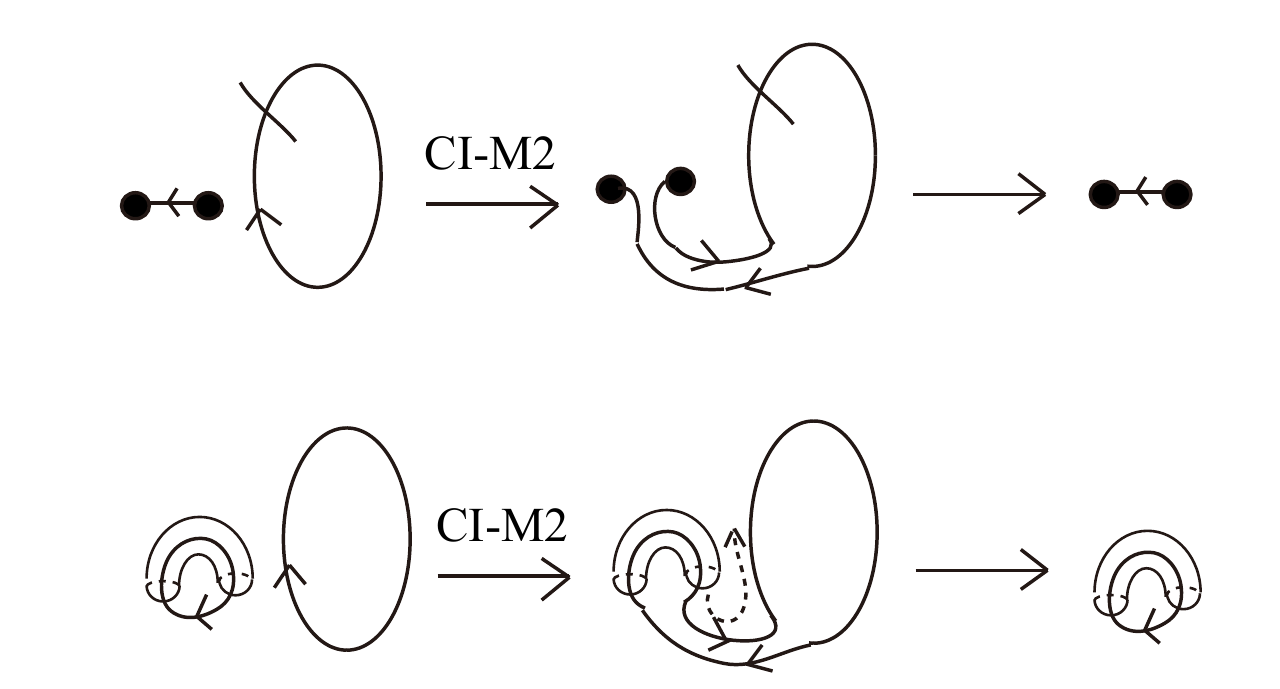}
\caption{Elimination of a chart loop by a free edge or a 1-handle with a chart loop.}
\label{fig5}
\end{figure}
\end{proof}

\begin{lemma}\label{lem3-3}
We have  
\begin{eqnarray}
&& f_i+h(\sigma_{i+1},e) \sim f_{i+1}+h(\sigma_i, e), \label{eq4-1}\\
&& f_i+h(e, \sigma_{i+1}) \sim f_{i+1}+h(e, \sigma_i).\label{eq4-2}
\end{eqnarray}
\end{lemma}

\begin{remark}
Since we assume that 1-handles are trivial, $h(e, \sigma_i) \sim h(\sigma_i, e)$ \cite[Lemma 4.4]{N5}, and (\ref{eq4-2}) directly follows from (\ref{eq4-1}). However, we give a proof avoiding the use of $h(e, \sigma_i) \sim h(\sigma_i, e)$ so that we can generalize the result to branched covering surface-knots with non-trivial 1-handles. We use (\ref{eq4-2}) to show Theorem \ref{thm1-8} in Section \ref{sec4}. 
\end{remark}

\begin{proof}[Proof of Lemma \ref{lem3-3}]
We show (\ref{eq4-1}). Assume that we have $f_i+h(e, \sigma_{i+1})$. 
By Fig. \ref{fig6}, $h(e, \sigma_{i+1})$ is equivalent to a 1-handle $h(\sigma_i, e)$ surrounded by two parallel chart loops such that the inner loop is of label $i+1$ and the outer loop is of label $i$. By Lemma \ref{lem2-3}, using $f_i$, we eliminate the outer loop so that $f_i+h(e, \sigma_{i+1})$ is deformed to the union of $f_i$ and $h(\sigma_i, e)$ surrounded by a chart loop $\rho_1$ of label $i+1$. By a CI-M2 move, we move $f_i$ into the region surrounded by the chart loop $\rho_1$. The free edge $f_i$ becomes surrounded by a chart loop of the label $i+1$, which deforms to a free edge with the label $i+1$, $f_{i+1}$, surrounded by a chart loop $\rho_2$ of the label $i$; see Figs. \ref{fig7} and \ref{fig8}. By Lemma \ref{lem2-3}, we eliminate the chart loop $\rho_2$ by using $h(\sigma_i, e)$. Then, we have $f_{i+1}+h(\sigma_i,e)$ surrounded by $\rho_1$, the chart loop with the label $i+1$. We eliminate $\rho_1$ by using $f_{i+1}$, and we have (\ref{eq4-1}). 

\begin{figure}[ht]
\centering
\includegraphics*[width=13cm]{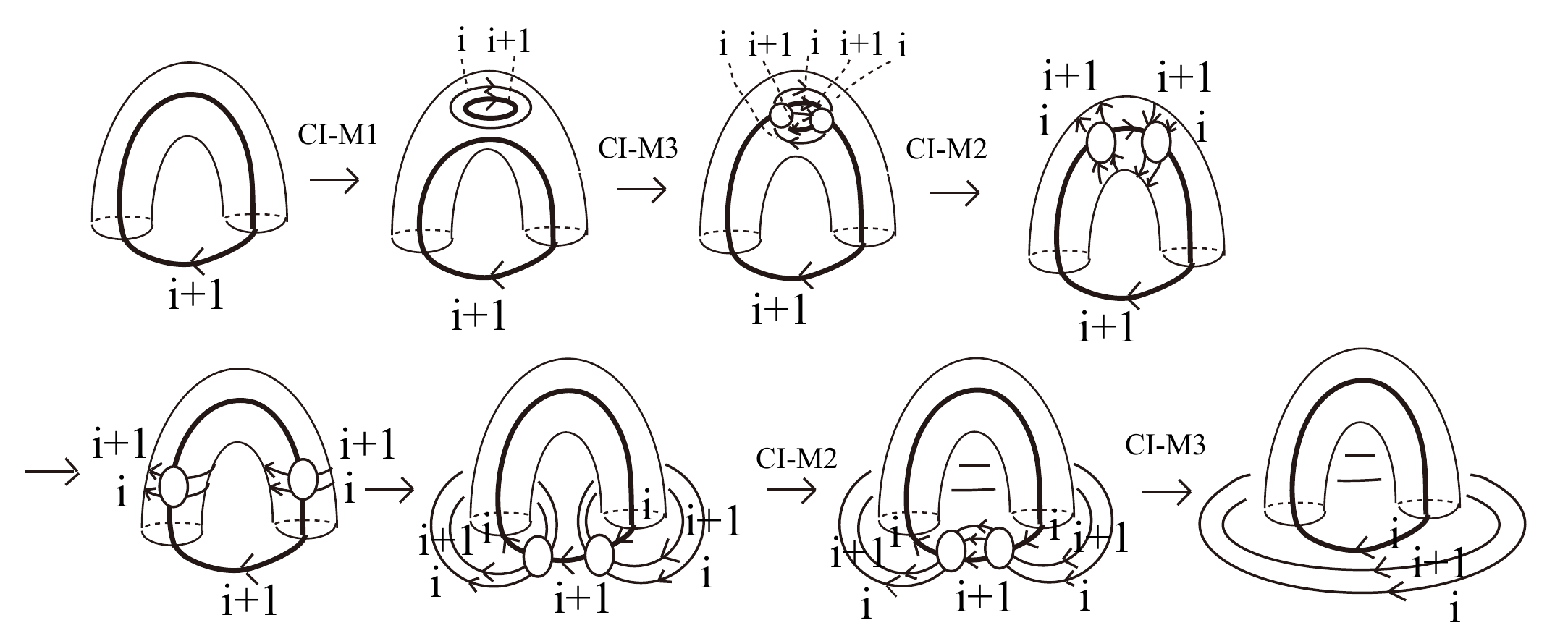}
\caption{$h(\sigma_{i+1}, e)$ is equivalent to $h(\sigma_i, e)$ surrounded by two parallel chart loops.}
\label{fig6}
 \end{figure}

\begin{figure}[ht]
\centering
\includegraphics*[height=2.5cm]{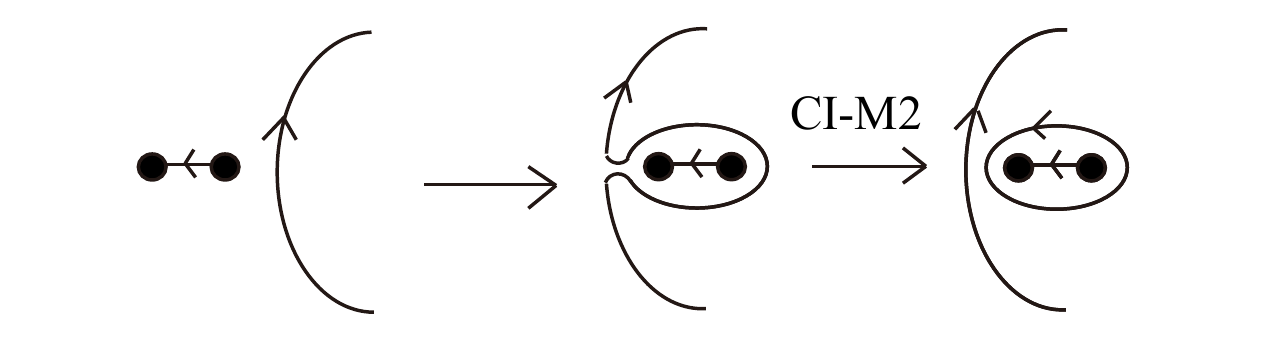}
\caption{Moving a free edge through a chart edge. The free edge becomes surrounded by a chart loop. We omit labels of chart edges. The orientations are for example.}
\label{fig7}
\end{figure}

\begin{figure}[ht]
\centering
\includegraphics*[width=13cm]{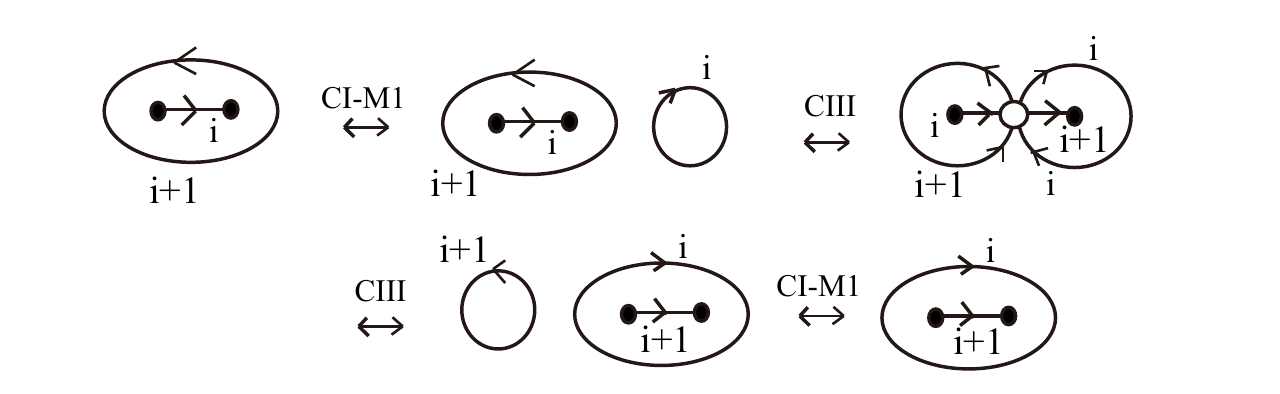}
\caption{Changing the labels of free edges surrounded by a loop.}
\label{fig8}
 \end{figure}

We show (\ref{eq4-2}). Assume that we have $f_i+h(e, \sigma_{i+1})$. By Lemma \ref{lem2-3}, using $f_i$ to make a chart loop with the label $i$, $f_i+h(e, \sigma_{i+1})$ deforms to the union of $f_i$ and $h(e, \sigma_{i+1})$ surrounded by a chart loop of the label $i$. As indicated in Fig. \ref{fig9}, by CI-M2 moves, this deforms to $h(e, \sigma_i \sigma_{i+1} \sigma_i^{-1})$. Since $\sigma_i \sigma_{i+1} \sigma_i^{-1}$ is equivalent to $\sigma_{i+1}^{-1} \sigma_i \sigma_{i+1}$, $h(e, \sigma_i \sigma_{i+1} \sigma_i^{-1})$ deforms to $h(e, \sigma_{i+1}^{-1} \sigma_i \sigma_{i+1})$ by CI-M3 and CI-M1 moves. By an inverse process similar to that indicated in Fig. \ref{fig9}, applying CI-M2 moves, this deforms to $h(e, \sigma_i)$ surrounded by a chart loop $\rho_1$ of the label $i+1$ with the clockwise orientation. By the same argument as in the case (\ref{eq4-1}), we move $f_i$ into the region surrounded by the chart loop $\rho_1$, and then it deforms to a free edge $f_{i+1}$ surrounded by a chart loop $\rho_2$ of the label $i$ with the clockwise orientation. Since $h(e, \sigma_i) \sim h(e, \sigma_i^{-1})$, as indicated in Fig. \ref{fig10}, by a CI-M2 move between $\rho_2$ and the chart loop of the label $i$ of $h(e, \sigma_i^{-1})$, and moving $f_{i+1}$ along the 1-handle, we have $f_{i+1}+h(e, \sigma_i^{-1})$ surrounded by $\rho_1$. We eliminate the chart loop $\rho_1$ by using $f_{i+1}$, and together with $h(e, \sigma_i^{-1}) \sim h(e, \sigma_i)$, we have (\ref{eq4-2}). 
\end{proof}

\begin{figure}[ht]
\centering
\includegraphics*[height=2.5cm]{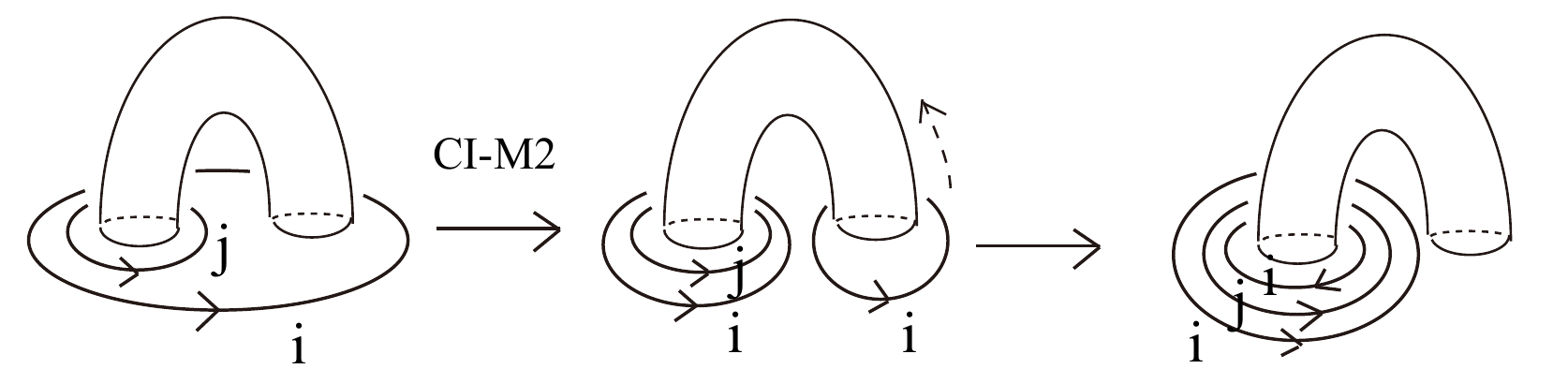}
\caption{$h(e, \sigma_j)$ surrounded by a chart loop of the label $i$ deforms to $h(e, \sigma_{i} \sigma_j \sigma_{i}^{-1})$. }
\label{fig9}
 \end{figure}

\begin{figure}[ht]
\centering
\includegraphics*[height=3cm]{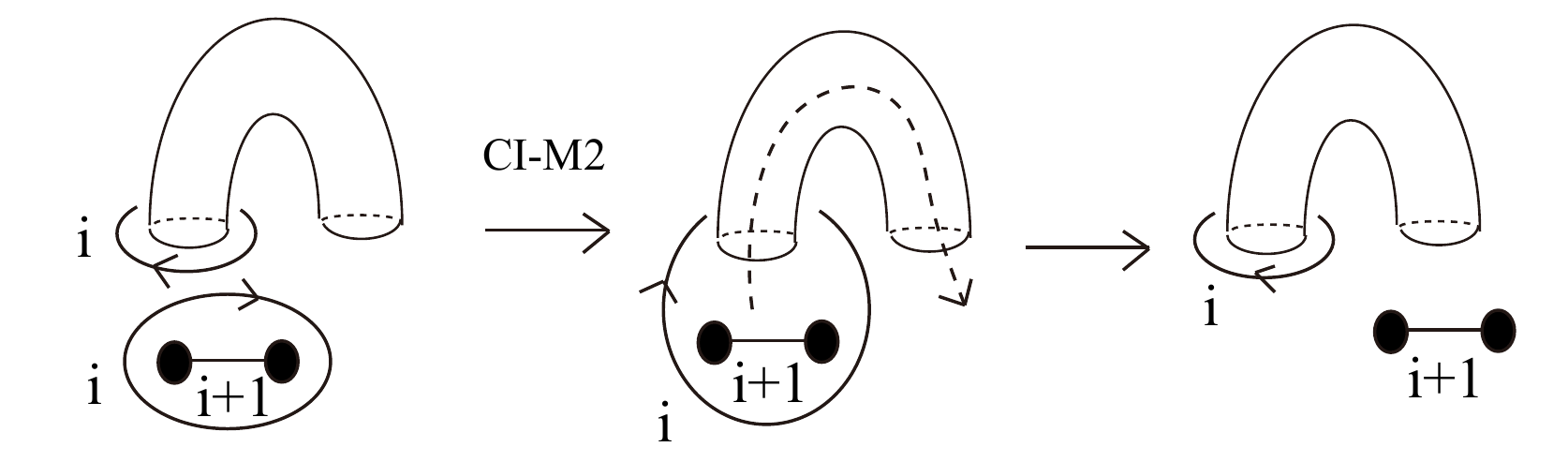}
\caption{Elimination of the chart loop surrounding the free edge. }
\label{fig10}
 \end{figure}

\begin{proof}[Proof of Theorem \ref{thm4-4}]
Let $(F, \Gamma)$ be a branched covering surface-knot such that $b(\Gamma)>0$ and $w(\Gamma)=0$. The chart consists of a finite number of free edges and chart loops. 
Choose a free edge $f$, and let $i$ be the label of $f$. Then 
add $N-2$ 1-handles $\sum_{j\neq i} h(\sigma_j, e)$ to a neighborhood of $f$. 
By Lemma \ref{lem3-3} (\ref{eq4-1}), $f$ deforms to be of any label. 
Hence, by Lemma \ref{lem2-3}, applying a CI-M2 move to $f$ and a chart loop $\rho$ which has $f$ in the neighborhood, and fixing one of the black vertex of $f$ and moving the other black vertex and applying CII-moves if necessary, we eliminate $\rho$ and the remaining chart is unchanged. Applying this process to every chart loop except those on the added 1-handles, we have 
free edges and 1-handles in the form $h(\sigma_i, e)$, 
which is a simplified form. Hence $u_w(F, \Gamma) \leq N-2$ and $u(F, \Gamma) \leq N-2$. 
\end{proof}

\begin{figure}[ht]
\centering
\includegraphics*[width=13cm]{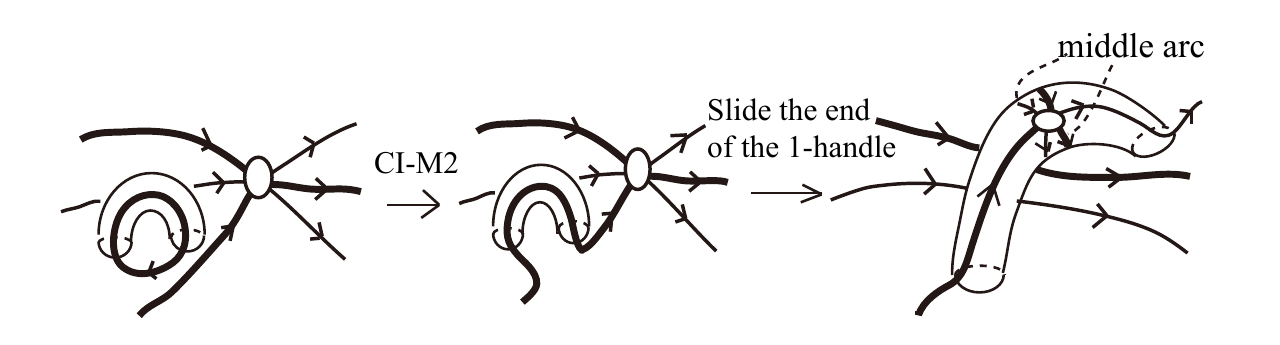}
\caption{Collecting a white vertex on a 1-handle. See Fig. \ref{fig12} for deformations when we slide the end of the 1-handle. For simplicity, we omit labels of chart edges.}
\label{fig11}
\end{figure}
\begin{figure}[ht]
\centering
\includegraphics*[width=13cm]{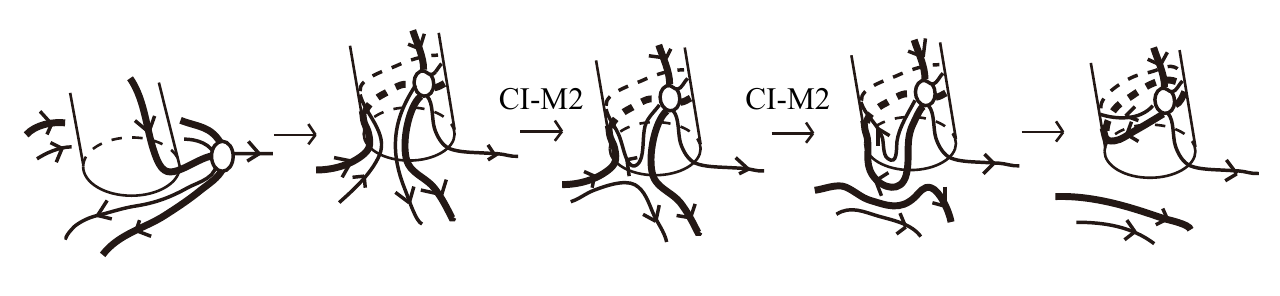}
\caption{Sliding the end of a 1-handle. For simplicity, we omit labels of chart edges.}
\label{fig12}
\end{figure}

\begin{proof}[Proof of Theorem \ref{thm1-16}]
Let $i$ be the label of a non-middle arc of a white vertex. 
We add a 1-handle $h(\sigma_i, e)$ to a neighborhood of a non-middle arc of label $i$ of each white vertex and slide an end of $h(\sigma_i, e)$ to collect the white vertex as indicated in Fig. \ref{fig11} (see also the proof of \cite[Theorem 1.6]{N5}). Then, all white vertices are on 1-handles, and hence all middle arcs are contained in chart edges parallel to cocores, each connected with one white vertex at endpoints; see the rightmost figure in Fig. \ref{fig11}. Hence, any edge connected with a black vertex is, on the other endpoint, connected to another black vertex or a non-middle arc of a white vertex. It follows that applying CIII-moves and CII-moves if necessary, the black vertices become endpoints of free edges. 
The result is the union of free edges and a chart such that there are no black vertices and white vertices are on 1-handles as in the rightmost figure in Fig. \ref{fig11}. The rest of the argument is similar to the proof of Theorem \ref{thm4-4} as follows. Since $b(\Gamma)>0$, we have 
a positive number of free edges. We choose one free edge $f$, and let $i$ be the label of $f$.  Then we add $N-2$ 1-handles
$\sum_{j\neq i} h(\sigma_j,e)$ to a neighborhood of $f$. The free edge $f$ deforms to be of any label (Lemma \ref{lem3-3}). Hence we eliminate chart loops which has $f$ in the neighborhood (Lemma \ref{lem2-3}). The edges connected to white vertices are on the other endpoints also connected to white vertices on 1-handles. Let $\rho$ be one connected component of the union of diagonal edges connecting distinct white vertices, where we assume that two edges forming one crossing are not connected. Since $\rho$ consists of non-middle arcs, when $f$ is in a neighborhood of $\rho$, applying a CI-M2 move and CIII-moves, and CII-moves if necessary, we eliminate $\rho$ and the connected white vertices. Applying these moves while fixing one black vertex, the result becomes the union of chart loops parallel to cocores, and the other chart is unchanged. Repeating these processes, we can eliminate all chart edges and vertices except free edges and chart loops on $\sum_{j\neq i} h(\sigma_j,e)$. Thus $u_w(F, \Gamma) \leq w(\Gamma)+N-2$ and $u(F, \Gamma) \leq w(\Gamma)+N-2$. 
\end{proof}

\section{Proof of Theorem \ref{thm1-8}}\label{sec4}

\begin{lemma}\label{lem4-1}
For $|i-j|=1$, we have
\begin{eqnarray}
&& f_i+h(\sigma_{j}, \sigma_k^\epsilon) \sim f_{j}+h(e, \sigma_i\sigma_k^\epsilon), \label{eq4-10}
\\
&& f_i+h(\sigma_k, \sigma_{j}^\epsilon) \sim f_{j}+h(\sigma_i \sigma_k, e), \label{eq4-11}
\end{eqnarray}
where $|j-k|>1$, $i,j,k\in \{1, \ldots, N-1\}$ and $\epsilon \in \{+1, -1\}$. 
\end{lemma}

 \begin{lemma}\label{lem4-1-2}
For $|i-j|=1$, we have 
\begin{eqnarray}
&& f_i+f_i+f_{j} \sim f_i+f_{j}+f_{j}, \\
&& f_i+f_i+h(\sigma_j, e) \sim f_i+f_{j}+h(e,e), \label{eq4-4y}\\
&& f_i+f_i+h(\sigma_{j}, \sigma_k^\epsilon) \sim f_i+f_{j}+h(e,\sigma_k^\epsilon), \\
&& f_i+f_i+h(e,\sigma_j^\epsilon) \sim f_i+f_{j}+h(e,e), \label{eq4-4s}\\
&& f_i+f_i+h(\sigma_k, \sigma_{j}^\epsilon) \sim f_i+f_{j}+h(\sigma_k, e), 
\end{eqnarray}
where $|j-k|>1$, $i,j,k \in \{1, \ldots, N-1\}$ and $\epsilon \in \{+1, -1\}$. 
\end{lemma}

\begin{lemma}\label{lem4-2}
Let $N \geq 4$. For $|i-j|=1$, We have 
\begin{equation}
f_i+h(\sigma_k, \sigma_{j}^\epsilon)+h(e,e) \sim f_{j}+h(\sigma_k, e)+h(\sigma_i,e), 
\end{equation}
where $|j-k|>1$, $i,j,k \in \{1, \ldots, N-1\}$ and $\epsilon \in \{+1, -1\}$. 
\end{lemma}

\begin{lemma}\label{lem4-3}
Let $N \geq 4$. We have 

\begin{equation}
f_i+f_{i+1}+h(e, \sigma_{i+2}^\epsilon)\sim f_i+f_{i+1}+h(\sigma_{i+2}, e),
\end{equation}
where $i \in \{1, \ldots, N-3\}$ and $\epsilon \in \{+1, -1\}$. 
\end{lemma}

\begin{lemma}\label{lem4-4}
We have 
\begin{eqnarray}
f_i+h(e, \sigma_{i+1}^\epsilon)+h(e, \sigma_{i+2}^\delta) &\sim&  f_i+h(\sigma_{i+1}, e)+h(\sigma_{i+2}, e), \label{eq4-8a}\\
f_i+h(e, \sigma_{i+1}^\epsilon)+h(\sigma_{i+2}, e)
&\sim& f_i+h(\sigma_{i+1}, e)+h(\sigma_{i+2}, e), \label{eq4-9a}\\
 f_i+h(\sigma_{i+1}, e)+h(e, \sigma_{i+2}^\delta)
&\sim& f_i+h(\sigma_{i+1}, e)+h(\sigma_{i+2}, e) \label{eq4-10a},
\end{eqnarray}
where $i \in \{1, \ldots, N-3\}$ and $\epsilon, \delta \in \{+1, -1\}$. 
\end{lemma}

\begin{proof}[Proof of Theorem \ref{thm1-8}]
Since a crossing of a chart does not exist for $N \leq 3$, a weak simplified form is a simplified form when $N \leq 3$. Thus it suffices to show the result for $N \geq 4$. 
Let $(F, \Gamma)$ be a branched covering surface-knot of degree $N \geq 4$ in a weak simplified form with $b(\Gamma)>0$. 
By Lemma \ref{lem2-3} using free edges and CII-moves, we eliminate chart loops of the same label with those of free edges. 
By Assumption \ref{assump1}, there are edges of all labels.  
By Lemma \ref{lem4-1}, if the branched covering surface-knot contains $f_i+h(\sigma_{j}, \sigma_k^\epsilon)$ ($|i-j|=1$, $|j-k|>1$), then we deform this to $f_{j}+h(e, \sigma_i \sigma_k^\epsilon)$ and using $f_j$, we eliminate chart loops of the label $j$, and  Lemma \ref{lem4-1} again, we deform $f_{j}+h(e, \sigma_i \sigma_k^\epsilon)$ to the original form $f_i+h(\sigma_{j}, \sigma_k^\epsilon)$. By similar processes, together with Lemma \ref{lem3-3}, we 
deform $(F, \Gamma)$ to a form consisting of free edges and 1-handles such that  and the labels of the chart loops are mutually distinct and do not contain those of free edges. If $h(F)=0$, then we have a simplified form. Assume that $h(F) \neq 0$. 
By Lemma \ref{lem4-1-2}, we deform $(F, \Gamma)$ so that free edges have as many labels as possible. We deform $(F, \Gamma)$ to the form such that the number of the set of labels is $\max \{f(\Gamma), N-1\}$. 
If $N-f(\Gamma)-h(F)-1 > 0$, then add $N-f(\Gamma)-h(F)-1$ copies of $h(e,e)$. 
By Lemma \ref{lem4-2}, we deform $(F, \Gamma)$ to the form such that each 1-handle has at most one chart loop. 
By Lemmas \ref{lem4-3} and \ref{lem4-4}, together with Lemma \ref{lem3-3}, we deform each 1-handle to the form $h(\sigma_i, e)$. Thus we have a simplified form, and $u(F, \Gamma) \leq \max \{0, N-f(\Gamma)-h(F)-1\}$.
\end{proof}

\subsection{Proofs of Lemmas \ref{lem4-1}--\ref{lem4-4}}

\begin{proof}[Proof of Lemma \ref{lem4-1}]
We show (\ref{eq4-10}). 
Assume that $f_i+h(\sigma_{j}, \sigma_k^\epsilon)$, where $|i-j|=1$ and $|j-k|>1$.  
By CI-M2 moves, $h(\sigma_{j}, \sigma_k^\epsilon)$ deforms to a form with two white vertices as indicated in the rightmost figure in Fig. \ref{fig14}. We apply a CI-M2 move between $f_i$ and an arc with the label $i$ of the 1-handle as in the leftmost figure in Fig. \ref{fig15}. Then, applying CIII-moves, a CII-move and a CI-M2 move as in Fig. \ref{fig15}, we have $f_{j}+h(e, \sigma_{i}\sigma_k^\epsilon)$, which implies the required result. The other relation (\ref{eq4-11}) is shown similarly. 
\end{proof}

\begin{figure}[ht]
\centering
\includegraphics*[height=3cm]{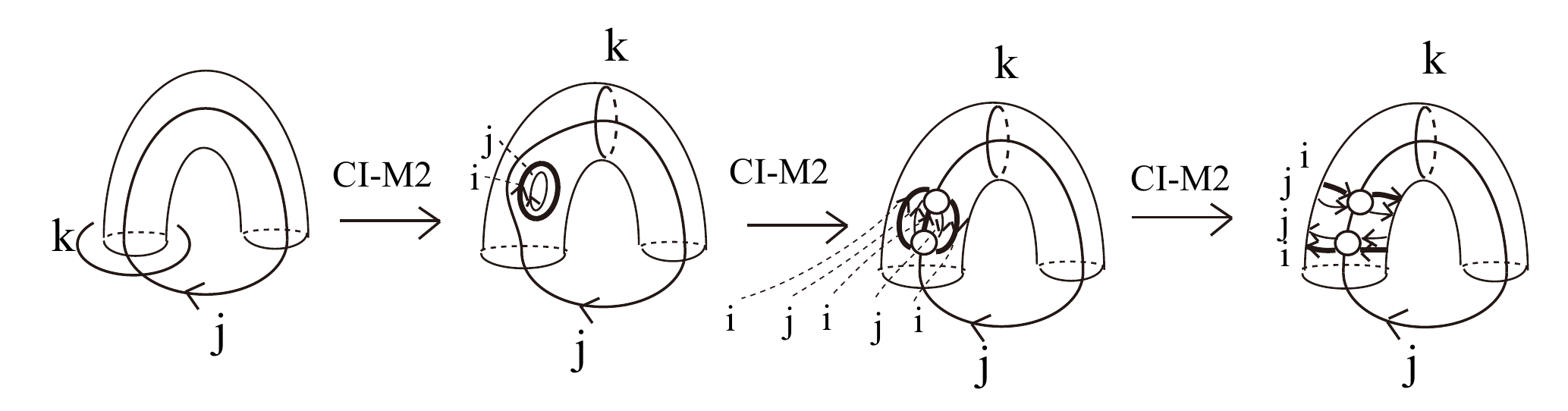}
\caption{The 1-handle $h(\sigma_{j}, \sigma_k^\epsilon)$ is equivalent to a 1-handle with two white vertices, where $|i-j|=1$ and $|j-k|>1$.}
\label{fig14}
 \end{figure}

\begin{figure}[ht]
\centering
\includegraphics*[height=6cm]{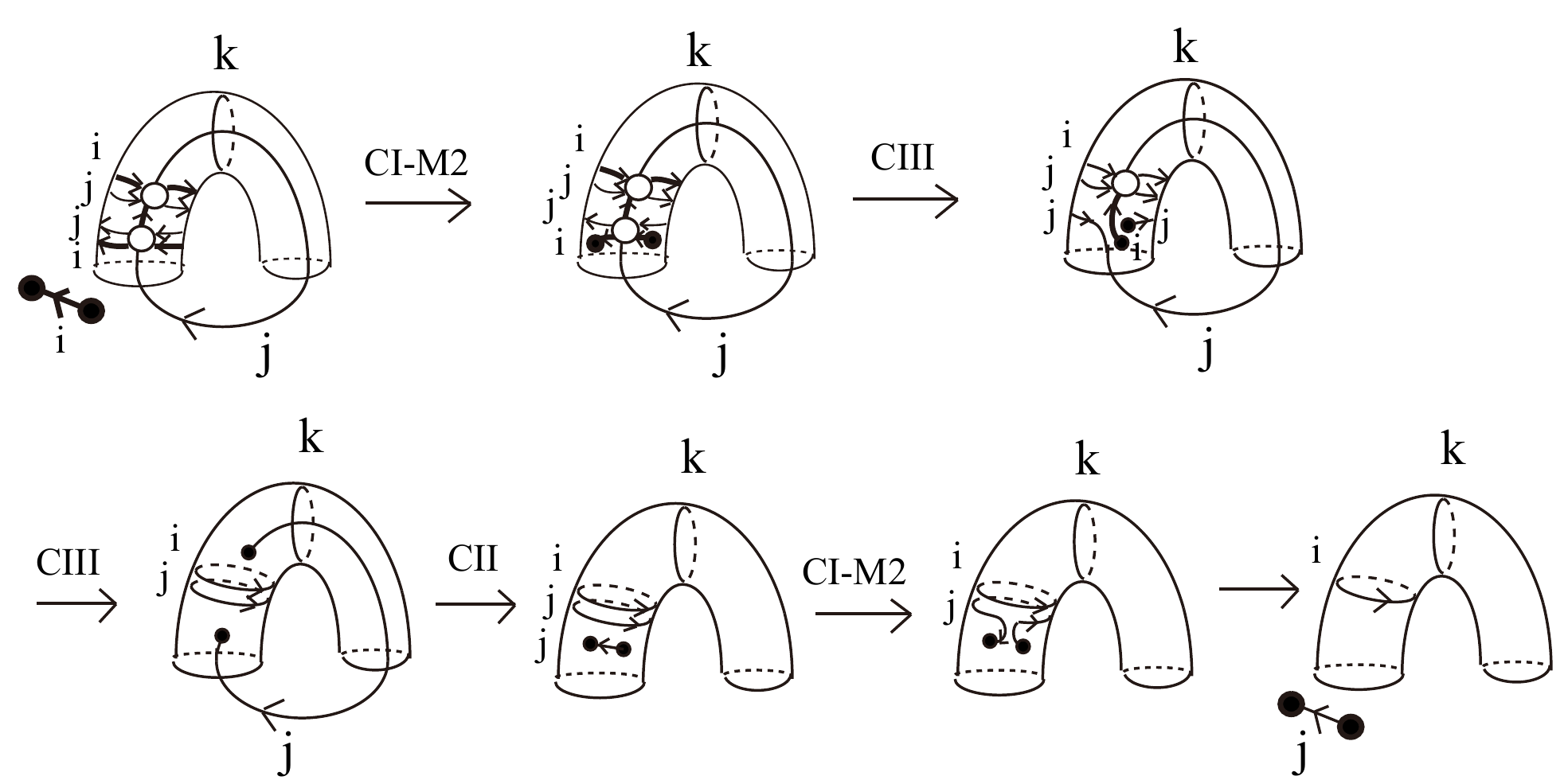}
\caption{$f_{i}+h(\sigma_{j}, \sigma_k^\epsilon) \sim f_{j}+h(e,\sigma_{i}\sigma_k^\epsilon)$, where $|i-j|=1$ and $|j-k|>1$. }
\label{fig15}
 \end{figure}

In the proof of \cite[Lemma 7.2]{N5}, we showed that  
$f_i+h(\sigma_i, e)+h(\sigma_{j}, e) \sim f_{j} +h(\sigma_i, e)+h(\sigma_{j}, e)$, 
for $|i-j|=1$, $i,j \in \{1,\ldots, N-1\}$. We show Lemma \ref{lem4-1-2} by a similar method (see also \cite{Kamada02}). The relations (\ref{eq4-4y}) and (\ref{eq4-4s}) are also shown from Lemmas \ref{lem3-1}--\ref{lem3-3}.
 
\begin{proof}[Proof of Lemma \ref{lem4-1-2}]
Assume that we have two copies of $f_i$, and one element of $\{f_j, h(\sigma_j,e), h(\sigma_j, \sigma_k^\epsilon), h(e, \sigma_j^\epsilon), h(\sigma_k, \sigma_j^\epsilon) \}$, denoted by $h$, where $|i-j|=1$ and $|j-k|>1$. Moving one of $f_i$ through an arc of the label $j$ as indicated in Fig. \ref{fig16} (see also Fig. \ref{fig7}), we deform $f_i$ to a free edge with the label $i$ surrounded by a chart loop of the label $j$. By C-moves, this deforms to a free edge with the label $j$ surrounded by a chart loop $\rho$ of the label $i$; see Fig. \ref{fig8}. Applying a CI-M2 move between $\rho$ and the other $f_i$, we can eliminate $\rho$ (Lemma \ref{lem2-3}). Thus we have $f_i$, $f_j$ and $h$. By Lemma \ref{lem2-3}, using $f_j$ to eliminate the chart loop of the label $j$, we have the required result. 
\end{proof}
\begin{figure}[ht]
\centering
\includegraphics*[height=5cm]{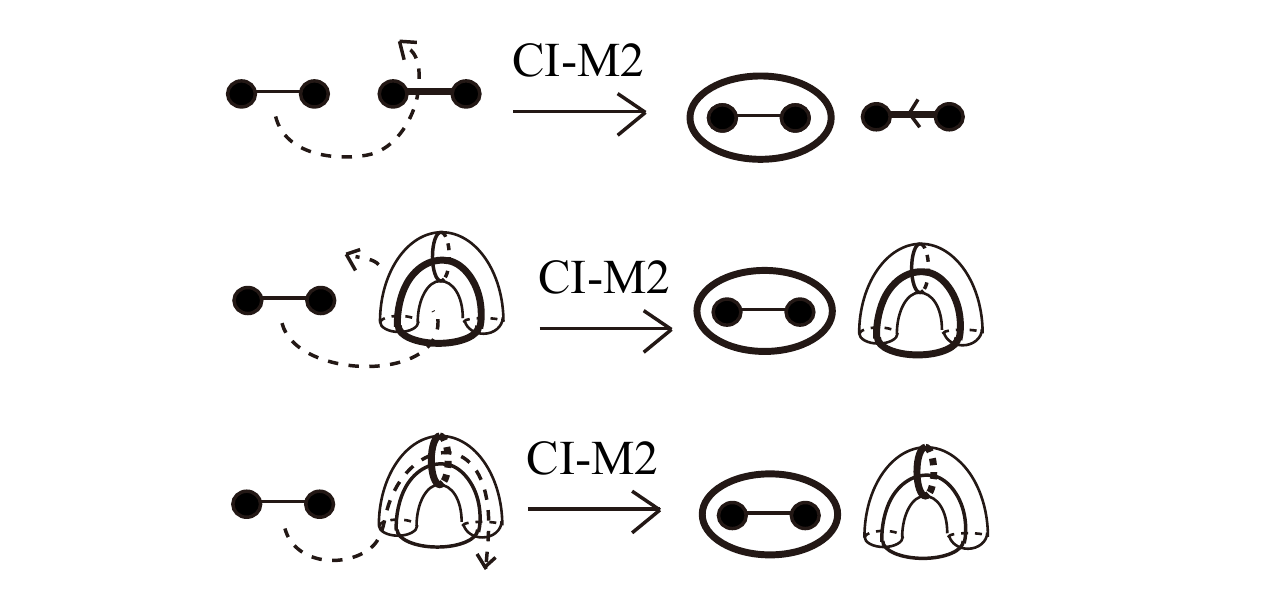}
\caption{Making a chart loop surrounding a free edge. We omit labels and orientations of chart edges. }
\label{fig16}
\end{figure}

\begin{lemma}\label{lem4-8}
We denote by $h_i$ a fixed 1-handle $h(\sigma_i, \sigma_k^\epsilon)$ or $h(\sigma_k, \sigma_i^\epsilon)$ for $|i-k|>1$ and $\epsilon \in \{+1, -1\}$. For $|i-j|=1$, 
\begin{eqnarray}
f_{i}+f_{j}+h_i \sim f_{j}+f_{j}+h_i, \\
f_i+h_i+h_{j} \sim f_{j}+h_i+h_{j}.
\end{eqnarray}

\end{lemma}

\begin{proof}
By a similar argument as in the proof of Lemma \ref{lem4-1-2}, we have the required result. 
\end{proof}

\begin{proof}[Proof of Lemma \ref{lem4-2}]
Assume that we have $f_i+h(\sigma_k, \sigma_{j}^\epsilon)+h(e,e)$, where $|i-j|=1$ and $|j-k|>1$. 
Applying Lemma \ref{lem2-3} using $f_i$, we make a chart loop of the label $i$ along the core loop of $h(e,e)$, and we have 
\[
f_i+h(\sigma_k, \sigma_{j}^\epsilon)+h(\sigma_i,e). 
\]
By Lemma \ref{lem4-8}, we have 
\[
f_{j}+h(\sigma_k, \sigma_{j}^\epsilon)+h(\sigma_i,e). 
\]
By Lemma \ref{lem2-3} using $f_{j}$, we eliminate the chart loop of the label $j$, and we have 
\[
f_{j}+h(\sigma_k, e)+h(\sigma_i,e), 
\]
which implies the required result. 
\end{proof}
 
\begin{proof}[Proof of Lemma \ref{lem4-3}]
Assume that we have $f_i+f_{i+1}+h(e, \sigma_{i+2}^\epsilon)$. 
Applying Lemma \ref{lem2-3} using $f_i$, we make a chart loop of the label $i$, and we have 
\[
f_i+f_{i+1}+h(\sigma_i, \sigma_{i+2}^\epsilon). 
\]
By Lemma \ref{lem4-8}, we have 
\[
f_{i+1}+f_{i+1}+h(\sigma_{i}, \sigma_{i+2}^\epsilon).
\]
By Lemma \ref{lem4-8} again, we have
\[
f_{i+1}+f_{i+2}+h(\sigma_{i}, \sigma_{i+2}^\epsilon).
\]
By Lemma \ref{lem2-3} using $f_{i+2}$, we eliminate the chart loop with the label $i+2$, and we have 
\[
f_{i+1}+f_{i+2}+h(\sigma_{i}, e). 
\]
Applying Lemma \ref{lem3-3} twice, we have
\[
f_{i}+f_{i+1}+h(\sigma_{i+2}, e), 
\]
which implies the required result. 
\end{proof}

\begin{proof}[Proof of Lemma \ref{lem4-4}]
We show (\ref{eq4-8a}). 
Assume that we have $f_i+h(e, \sigma_{i+1}^\epsilon)+h(e, \sigma_{i+2}^\delta)$. By Lemma \ref{lem2-3} using $f_i$, we make a chart loop with the label $i$, and we have 
\[
f_i+h(e, \sigma_{i+1}^\epsilon)+h(\sigma_i, \sigma_{i+2}^\delta).
\]
By Lemma \ref{lem4-8}, we have 
\[
f_{i+1}+h(e, \sigma_{i+1}^\epsilon)+h(\sigma_i, \sigma_{i+2}^\delta).
\]
By Lemma \ref{lem2-3} using $f_{i+1}$, we eliminate the chart loop with the label $i+1$, and we have 
\[
f_{i+1}+h(e,e)+h(\sigma_i, \sigma_{i+2}^\delta).
\]
By Lemma \ref{lem2-3} using $f_{i+1}$, we make a chart loop with the label $i+1$, and we have 
\[
f_{i+1}+h(\sigma_{i+1},e)+h(\sigma_i, \sigma_{i+2}^\delta).
\]
By Lemma \ref{lem4-8}, 
\[
f_{i+2}+h(\sigma_{i+1},e)+h(\sigma_i, \sigma_{i+2}^\delta).
\]
By Lemma \ref{lem2-3} using $f_{i+2}$, we eliminate the chart loop with the label $i+2$, and we have 
\[
f_{i+2}+h(\sigma_{i+1},e)+h(\sigma_i, e).
\]
Applying Lemma \ref{lem3-3} twice, we have
\[
f_{i}+h(\sigma_{i+1},e)+h(\sigma_{i+2}, e),
\]
which implies (\ref{eq4-8a}). 

We show (\ref{eq4-9a}).
Assume that we have $f_i+h(e, \sigma_{i+1}^\epsilon)+h(\sigma_{i+2}, e)$. By Lemma \ref{lem2-3} using $f_i$, we make a chart loop with the label $i$, and we have 
\[
f_i+h(e, \sigma_{i+1}^\epsilon)+h(\sigma_{i+2}, \sigma_i).
\]
By Lemma \ref{lem4-8}, we have 
\[
f_{i+1}+h(e, \sigma_{i+1}^\epsilon)+h(\sigma_{i+2}, \sigma_i).
\]
By Lemma \ref{lem2-3} using $f_{i+1}$, we eliminate the chart loop with the label $i+1$, and we have 
\[
f_{i+1}+h(e, e)+h(\sigma_{i+2}, \sigma_i). 
\]
By a similar argument as in the case (\ref{eq4-8a}), we have 
\[
f_{i}+h(\sigma_{i+1}, e)+h(\sigma_{i+2}, e), 
\]
which implies (\ref{eq4-9a}). 
The last relation (\ref{eq4-10a})
is shown similarly. 
\end{proof}

\section*{Acknowledgements}
The author would like to thank Professor Seiichi Kamada for his helpful comments. 
The author was partially supported by JSPS KAKENHI Grant Numbers 15H05740 and 15K17532.

\end{document}